%% file: main.tex
\documentclass[11pt]{article}

\usepackage[T1]{fontenc}
\usepackage[utf8]{inputenc}
\usepackage{lmodern}
\usepackage{amsmath,amssymb,amsthm,mathtools}
\usepackage{enumitem}
\usepackage{booktabs,tabularx}
\usepackage{tikz}
\usepackage{float}
\usepackage{microtype}
\usepackage[a4paper,margin=1in]{geometry}
\usepackage[hidelinks]{hyperref}

\allowdisplaybreaks

\newtheorem{theorem}{Theorem}[section]
\newtheorem{lemma}[theorem]{Lemma}
\newtheorem{corollary}[theorem]{Corollary}
\newtheorem{proposition}[theorem]{Proposition}
\theoremstyle{definition}

\theoremstyle{remark}

\input{metadata}
\hypersetup{
  pdftitle={Grundy Total Domination and Skew Zero Forcing in Cartesian Products of Paths and Cycles},
  pdfauthor={\ManuscriptAuthorName}
}

\title{Grundy Total Domination and Skew Zero Forcing in\\
Cartesian Products of Paths and Cycles}
\author{\ManuscriptAuthorName\\[0.25em]
\small \ManuscriptAuthorDepartment, \ManuscriptAuthorInstitution\\
\small \ManuscriptAuthorLocation\\
\small \texttt{\ManuscriptAuthorEmail}}
\date{}

\begin{document}

\maketitle
\input{00_abstract}
\input{01_introduction}
\input{02_preliminaries}
\input{03_spectral_tools}
\input{04_path_path}
\input{05_path_cycle}
\input{06_cycle_cycle}
\input{07_conclusion}
\input{08_declarations}

\bibliographystyle{plain}
\bibliography{references}

\end{document}

%% file: metadata.tex
\newcommand{\ManuscriptAuthorName}{Fei-Huang Chang}
\newcommand{\ManuscriptAuthorDepartment}{Department of Applied Mathematics}
\newcommand{\ManuscriptAuthorInstitution}{National Dong Hwa University}
\newcommand{\ManuscriptAuthorLocation}{Hualien 974301, Taiwan}
\newcommand{\ManuscriptAuthorEmail}{cfh@gms.ndhu.edu.tw}

\newcommand{\ManuscriptFundingAgency}{National Science and Technology Council of Taiwan}
\newcommand{\ManuscriptGrantNumber}{NSTC 114-2115-M-259-005}

%% file: 00_abstract.tex
\begin{abstract}
We determine the Grundy total domination number and the skew zero forcing
number for every Cartesian product of two paths or cycles, thereby closing
the nonmatching bounds previously known for these rectangular,
cylindrical, and toroidal families.  For
$2\le a\le b$,
\[
 Z_-(P_a\square P_b)
 =a-\mathbf1_{\{a\ {\rm odd},\,b\ {\rm even}\}}.
\]
For $p\ge2$ and $c\ge3$,
\[
 Z_-(P_p\square C_c)=
 \begin{cases}
 \min\{p,c\},&c\text{ is odd},\\
 \min\{2p,c\},&c\text{ is even}.
 \end{cases}
\]
Finally, for $3\le a\le b$,
\[
 Z_-(C_a\square C_b)=
 \begin{cases}
 2a-1,&a=b\text{ odd},\\
 2a,&a=b\text{ even},\\
 \min\{b,2a\},&a<b,\ a\text{ odd},\\
 a,&a<b,\ a\text{ even},\ b\text{ odd},\\
 2a,&a<b,\ a,b\text{ even}.
 \end{cases}
\]
In every case,
$\gamma_{\mathrm{gr}}^t(G)=|V(G)|-Z_-(G)$.  The path-containing lower
bounds use Kronecker differences of skew-symmetric or hollow symmetric
factor matrices, and they also determine maximum skew nullity and minimum
skew rank.  They further produce an infinite family for which maximum
skew nullity is strictly smaller than skew zero forcing.  The new
cycle--cycle lower bounds use a cyclic column-defect estimate, applied to
complete columns when the shorter cycle is odd and, after a one-sided
bipartite reduction, to half-columns when both cycles are even.  Explicit
forcing constructions give matching upper bounds throughout.
\end{abstract}



%% file: 01_introduction.tex
\section{Introduction}\label{sec:introduction}

Let $G$ be a finite simple graph without isolated vertices.  A sequence
$S=(v_1,\ldots,v_k)$ of distinct vertices of $G$ is a legal
open-neighborhood sequence if
\[
N_G(v_i)\setminus\bigcup_{j=1}^{i-1}N_G(v_j)\ne\varnothing
\qquad\text{for every }1\le i\le k.
\]
A maximum-length legal open-neighborhood sequence is called a Grundy
total dominating sequence, and its length is the Grundy total domination
number $\gamma_{\mathrm{gr}}^t(G)$, introduced by
Bre\v{s}ar, Henning, and Rall~\cite{BresarHenningRallTotalSequences}.
Because $G$ has no isolated vertices, every maximum-length legal sequence
satisfies $\bigcup_iN_G(v_i)=V(G)$: otherwise a neighbor of an uncovered
vertex could be appended.
This parameter measures the largest number of legal choices in a greedy
total domination process.  We study it on Cartesian products of paths and
cycles, which form rectangular grids, cylindrical grids, and toroidal
grids.  Neighborhood-sequence variants and their zero-forcing connections
were developed further in~\cite{BresarEtAlGrundyZeroForcing}, while the
closed-neighborhood analogue was studied on grid-like and toroidal graphs
in~\cite{BresarEtAlGridToroidal}.  For Grundy total domination, the three
Cartesian-product families considered here were treated previously by
Bre\v{s}ar et al.~\cite{BresarEtAlProductGraphs}.

Our proofs use the skew zero forcing number.  Under the skew color-change
rule, a vertex, whether blue or white, may force its unique white neighbor
to become blue.  The minimum cardinality of an initial blue set from which
all vertices can be colored blue is denoted by $Z_-(G)$.  Lin proved that
\[
\gamma_{\mathrm{gr}}^t(G)=|V(G)|-Z_-(G)
\]
for every finite simple graph~\cite[Theorem~2.2]{LinZeroForcingGrundy}.
Thus, determining either parameter determines the other.  We formulate
the problem in terms of $\gamma_{\mathrm{gr}}^t$, organize the proofs
through $Z_-$, and state every main result for both parameters.  For
background on the inverse-matrix problems behind zero forcing and its
variants, see~\cite{HogbenLinShaderBook}.

For $3\le a\le b$ and $p,c\ge3$, the Cartesian-product results of
Bre\v{s}ar et al.~\cite[Theorems~5.1 and~5.2]{BresarEtAlProductGraphs}
give
\begin{align*}
ab-a
&\le \gamma_{\mathrm{gr}}^t(P_a\square P_b)
\le ab-\left\lfloor\frac a2\right\rfloor,\\
pc-\min\{2p,c\}
&\le \gamma_{\mathrm{gr}}^t(P_p\square C_c)
\le pc-\min\left\{p,\left\lceil\frac c2\right\rceil\right\},\\
ab-2a
&\le \gamma_{\mathrm{gr}}^t(C_a\square C_b)
\le ab-a.
\end{align*}
The same paper records the known square-path formula
$\gamma_{\mathrm{gr}}^t(P_s\square P_s)=s^2-s$ and, for odd $s$,
derives $\gamma_{\mathrm{gr}}^t(C_s\square C_s)=(s-1)^2$.
The displayed intervals leave the exact values undetermined for general
rectangular grids, cylindrical grids, and most toroidal grids.

Closing these intervals requires more than sharpening one common bound.
The spectra of skew-symmetric matrices impose parity restrictions, cyclic
forcing constructions must close across a seam, and the toroidal lower
bounds depend on how a legal sequence first completes a layer.  These
obstructions occur in different parity regimes and lead naturally to the
combination of linear-algebraic certificates, explicit forcing dynamics,
and direct footprint counting used below.

The matrix bounds associated with skew zero forcing provide the lower
bounds on $Z_-$ used in this paper.  Let $M_-(G)$ be the maximum nullity
over the real skew-symmetric matrices described by $G$, and let $M_0(G)$
be the corresponding maximum over real hollow symmetric matrices.  The
minimum skew rank is
$\operatorname{mr}^-(G)=|V(G)|-M_-(G)$.  The known inequalities
\[
M_-(G)\le Z_-(G)
\qquad\text{and}\qquad
M_0(G)\le Z_-(G)
\]
follow from skew minimum-rank theory and Lin's correspondence
\cite{IMAISUMinimumSkewRank,LinZeroForcingGrundy}.  If $A$ and $B$ are
matrices of orders $m$ and $n$, respectively, that are diagonalizable over
$\mathbb C$, then the Kronecker difference
\[
K=A\otimes I_n-I_m\otimes B
\]
has nullity
\[
\operatorname{nullity}K
=\sum_{\lambda\in\operatorname{supp}\sigma(A)\cap
\operatorname{supp}\sigma(B)}
  m_A(\lambda)m_B(\lambda),
\]
where $\operatorname{supp}\sigma(A)$ is the set of distinct eigenvalues
of $A$, and $m_A(\lambda)$ and $m_B(\lambda)$ denote algebraic
multiplicities.
Suitable factor patterns make $K$ describe a Cartesian product.  This
product-matrix mechanism has symmetric and skew-spectral precedents
\cite{BensonEtAlGraphProducts,CuiHouSkewCartesian}; our task is to
construct factor matrices whose common spectra have the multiplicities
required by the forcing bounds.  The inequalities $M_-(G)\le Z_-(G)$ and
$M_0(G)\le Z_-(G)$ need not be equalities in general; below we identify
both equality families and one infinite family with $M_-(G)<Z_-(G)$.
That strict gap is also concrete evidence that factor-matrix certificates
cannot supply all the lower bounds in the classification.  Conversely,
the later column-defect method exploits periodic layers and does not
determine the associated matrix-nullity parameters.  The two approaches
are therefore complementary rather than interchangeable.

The path--path result determines both parameters for every Cartesian
product of two nontrivial paths.
\begin{theorem}\label{thm:path-product}
Let $2\le a\le b$.  Then
\[
Z_-(P_a\square P_b)
=a-\mathbf 1_{\{a\text{ is odd and }b\text{ is even}\}},
\]
and consequently
\[
\gamma_{\mathrm{gr}}^t(P_a\square P_b)
=ab-a+\mathbf 1_{\{a\text{ is odd and }b\text{ is even}\}}.
\]
\end{theorem}
A boundary column gives the forcing upper bound $a$ in every case.  In the
odd--even branch, a two-sided checkerboard sweep saves one initial blue
vertex.  For the reverse inequality, we construct skew-symmetric path
matrices whose Kronecker difference has the required nullity.  The one-unit
exception reflects the parity restriction on the spectra of skew-symmetric
path matrices.  The same certificates also determine
$M_-(P_a\square P_b)$ and hence the minimum skew rank of every path product.

We next determine both parameters for every product of a path and a cycle.
\begin{theorem}\label{thm:pc-exact-skew-zero-forcing}
Let $p\ge2$ and $c\ge3$, and define
\[
q(p,c)=
\begin{cases}
\min\{p,c\},&c\text{ is odd},\\
\min\{2p,c\},&c\text{ is even}.
\end{cases}
\]
Then
\[
Z_-(P_p\square C_c)=q(p,c)
\]
and
\[
\gamma_{\mathrm{gr}}^t(P_p\square C_c)=pc-q(p,c).
\]
\end{theorem}
The elementary upper bounds $c$ and $2p$ follow by sweeping from one
boundary cycle layer and from two adjacent path layers, respectively.
When $c$ is odd and $p<c$, our smaller forcing set begins with a step-two
pattern in one boundary cycle layer and closes across the cyclic seam.
The matrix lower bound is skew-symmetric except when $p$ is even, $c$ is
odd, and $p\ge c$.  In that branch, the nullity parity of skew-symmetric
matrices requires a hollow symmetric certificate.  More precisely,
\[
M_-(P_p\square C_c)=c-1<c=Z_-(P_p\square C_c)=M_0(P_p\square C_c),
\]
whereas $M_-=Z_-$ in every other path--cycle branch.  Thus the minimum
skew rank is determined throughout the family.

The cycle--cycle family also admits a complete classification.
\begin{theorem}\label{thm:cc-exact}
Let $3\le a\le b$, and define
\[
\rho(a,b)=
\begin{cases}
2a-1,&a=b\text{ and }a\text{ is odd},\\
2a,&a=b\text{ and }a\text{ is even},\\
\min\{b,2a\},&a<b\text{ and }a\text{ is odd},\\
a,&a<b,\ a\text{ is even, and }b\text{ is odd},\\
2a,&a<b\text{ and }a,b\text{ are even}.
\end{cases}
\]
Then
\[
Z_-(C_a\square C_b)=\rho(a,b)
\qquad\text{and}\qquad
\gamma_{\mathrm{gr}}^t(C_a\square C_b)=ab-\rho(a,b).
\]
\end{theorem}
The odd-square case was known
\cite{BresarEtAlProductGraphs}.  The new lower bounds in the previously
unresolved odd-shorter and even-even branches use one common column-defect
estimate.  For a legal prefix $S$, the quantity
$|N(S)|-|S|$ is the cumulative footprint excess.  Writing the numbers of
unselected and undominated vertices column by column turns this defect
into a cyclic height charge bounded below by $\min\{b-1,2q\}$.  This is a
columnwise refinement of the arc-counting method of
Bre\v{s}ar et al.~\cite[Theorem~5.1]{BresarEtAlProductGraphs}.  When the
shorter cycle has odd order, take $q=a$ and examine the first completed
column; the equality case supplies the final unit needed for the bound
$\min\{b,2a\}$.  When both cycle orders are even, Lin's one-sided identity
reduces the problem to legal sequences supported on one bipartition class;
the same estimate, with $q=a/2$, gives a loss of at least $a$ in that
class and hence $Z_-\ge2a$.  Closed-diagonal and
alternating half-layer forcing sets provide the matching upper bounds.
If the shorter order is even and the longer one is odd, the half-layer
construction meets the known lower bound $a$.

The paper is organized as follows.  Section~\ref{sec:preliminaries}
introduces skew zero forcing and the relevant matrix parameters.
Section~\ref{sec:spectral-tools} develops the product-matrix tools and the
factor spectra used later.  Sections~\ref{sec:path-path} and
\ref{sec:path-cycle} prove Theorems~\ref{thm:path-product} and
\ref{thm:pc-exact-skew-zero-forcing}, respectively.
Section~\ref{sec:cycle-cycle} proves Theorem~\ref{thm:cc-exact} by the
cyclic column-defect method.
Section~\ref{sec:conclusion} summarizes the classifications and their
proof tools.

%% file: 02_preliminaries.tex
\section{Preliminaries}\label{sec:preliminaries}

All graphs in this paper are finite and simple.  For a graph $G$ and a
vertex $v$, let $N_G(v)$ denote the open neighborhood of $v$.  Unless a
statement explicitly says otherwise, graphs have no isolated vertices.
The path and cycle on $n$ vertices are denoted by $P_n$ and $C_n$,
respectively; $I_n$ denotes the identity matrix of order $n$, and
$\mathbf 1_{\mathcal P}$ is $1$ when the property $\mathcal P$ holds and
$0$ otherwise.
The Cartesian product $G\square H$ has vertex set $V(G)\times V(H)$, with
$(u,x)$ adjacent to $(v,y)$ if and only if either $u=v$ and $xy\in E(H)$,
or $x=y$ and $uv\in E(G)$.

In the skew zero forcing process, the vertices are colored blue or white.
If a vertex $u$, of either color, has exactly one white neighbor $v$, then
$u$ may force $v$; we write $u\to v$.  A set $B\subseteq V(G)$ is a
\emph{skew zero forcing set} if, starting with exactly the vertices of $B$
blue, repeated legal forces can color all vertices blue.  The minimum
cardinality of such a set is the \emph{skew zero forcing number} $Z_-(G)$.
Lin's correspondence gives
\begin{equation}\label{eq:grundy-skew-identity}
\gamma_{\mathrm{gr}}^t(G)=|V(G)|-Z_-(G).
\end{equation}

Let $G$ have vertex set $\{1,\ldots,n\}$.  We write
\begin{align*}
\mathcal S_-(G)
&=\{A=[a_{ij}]\in\mathbb R^{n\times n}:A^T=-A,
\ a_{ij}\ne0\Longleftrightarrow ij\in E(G)\text{ for }i\ne j\},\\
\mathcal S_0(G)
&=\{A=[a_{ij}]\in\mathbb R^{n\times n}:A^T=A,
\ a_{ii}=0,
\ a_{ij}\ne0\Longleftrightarrow ij\in E(G)\text{ for }i\ne j\}.
\end{align*}
Thus, $\mathcal S_-(G)$ is the class of real skew-symmetric matrices
described by $G$, whereas $\mathcal S_0(G)$ is the class of real hollow
symmetric matrices described by $G$.  Define
\[
M_-(G)=\max_{A\in\mathcal S_-(G)}\operatorname{nullity}A,
\qquad
M_0(G)=\max_{A\in\mathcal S_0(G)}\operatorname{nullity}A.
\]
The \emph{minimum skew rank} of $G$ is
\[
\operatorname{mr}^-(G)
=\min_{A\in\mathcal S_-(G)}\operatorname{rank}A
=n-M_-(G).
\]
The standard matrix bounds for skew zero forcing are
\begin{equation}\label{eq:matrix-skew-bounds}
M_-(G)\le Z_-(G)
\qquad\text{and}\qquad
M_0(G)\le Z_-(G);
\end{equation}
see \cite{IMAISUMinimumSkewRank,LinZeroForcingGrundy}.  For the
path-containing products, we obtain lower bounds on $Z_-$ by constructing
matrices in one of these two classes with prescribed nullity.  Neither
inequality in
\eqref{eq:matrix-skew-bounds} is an equality in general; equality in the
families treated below is established only when a matrix certificate and a
forcing construction meet at the same value.

For a square matrix $A$, let $\sigma(A)$ denote its spectrum over
$\mathbb C$, counted as a multiset when multiplicities matter, and let
$m_A(\lambda)$ denote the algebraic multiplicity of $\lambda$.  Every real
symmetric or real skew-symmetric matrix is diagonalizable over $\mathbb C$.

%% file: 03_spectral_tools.tex
\section{Product-matrix and spectral tools}\label{sec:spectral-tools}

The symmetric Cartesian-product construction based on a Kronecker
difference is standard in the minimum-rank literature; see, for example,
\cite{BensonEtAlGraphProducts}.  The skew-adjacency Kronecker-sum and its
eigenvalue rule also appear in \cite{CuiHouSkewCartesian}.  We record the
weighted skew-symmetric and hollow symmetric versions needed below, together
with the relevant multiplicity bookkeeping.  The first lemma below verifies
that the Kronecker difference has exactly the zero--nonzero pattern required
by the Cartesian product.  The second counts its kernel through common factor
eigenvalues.  The remaining task in each application is therefore to choose
factor matrices whose common eigenvalues have the required multiplicities.

\subsection{Kronecker differences}

\begin{lemma}\label{lem:kronecker-pattern}
Let $G$ and $F$ be finite simple graphs of orders $m$ and $n$, respectively,
and let
\[
K=A\otimes I_n-I_m\otimes B.
\]
If $A\in\mathcal S_-(G)$ and $B\in\mathcal S_-(F)$, then
$K\in\mathcal S_-(G\square F)$.  If $A\in\mathcal S_0(G)$ and
$B\in\mathcal S_0(F)$, then $K\in\mathcal S_0(G\square F)$.
\end{lemma}

\begin{proof}
Index the rows and columns of $K$ by $V(G)\times V(F)$.  For
$u,v\in V(G)$ and $x,y\in V(F)$, the corresponding entry is
\[
k_{(u,x),(v,y)}=a_{uv}\delta_{xy}-\delta_{uv}b_{xy},
\]
where $\delta$ is the Kronecker delta.  Since both $A$ and $B$ have zero
diagonal, every diagonal entry of $K$ is zero.

Let $(u,x)\ne(v,y)$.  If $x=y$, then
$k_{(u,x),(v,x)}=a_{uv}$; if $u=v$, then
$k_{(u,x),(u,y)}=-b_{xy}$; and if both coordinates differ, the entry is
zero.  Thus the nonzero off-diagonal entries occur precisely at the edges
of $G\square F$.  Finally,
\[
K^T=A^T\otimes I_n-I_m\otimes B^T.
\]
Hence $K^T=-K$ when both factors are skew-symmetric, and $K^T=K$ when
both factors are symmetric.  Together with the zero diagonal, this proves
both assertions.
\end{proof}

\begin{lemma}\label{lem:kronecker-nullity}
Let $A\in\mathbb C^{m\times m}$ and $B\in\mathbb C^{n\times n}$ be
diagonalizable.  Then
\[
\operatorname{nullity}(A\otimes I_n-I_m\otimes B)
=\sum_{\lambda\in\operatorname{supp}\sigma(A)\cap
\operatorname{supp}\sigma(B)}
m_A(\lambda)m_B(\lambda),
\]
where $\operatorname{supp}\sigma(A)$ denotes the set of distinct
eigenvalues of $A$ (with multiplicities suppressed), and similarly for
$B$.
\end{lemma}

\begin{proof}
Choose eigenbases $x_1,\ldots,x_m$ of $A$ and $y_1,\ldots,y_n$ of $B$,
where $Ax_i=\alpha_i x_i$ and $By_j=\beta_j y_j$.  The vectors
$x_i\otimes y_j$ form a basis of $\mathbb C^{mn}$, and
\begin{align*}
(A\otimes I_n-I_m\otimes B)(x_i\otimes y_j)
&=(Ax_i)\otimes y_j-x_i\otimes(By_j)\\
&=(\alpha_i-\beta_j)(x_i\otimes y_j).
\end{align*}
Thus a basis vector lies in the kernel exactly when
$\alpha_i=\beta_j$.  Counting these pairs for each common eigenvalue gives
the formula.
\end{proof}

\begin{corollary}\label{cor:product-spectral-certificate}
Let $G$ and $F$ be finite simple graphs without isolated vertices and of
orders $m$ and $n$, respectively.  Suppose that either
\[
A\in\mathcal S_-(G),\qquad B\in\mathcal S_-(F),
\]
or
\[
A\in\mathcal S_0(G),\qquad B\in\mathcal S_0(F).
\]
Then
\[
Z_-(G\square F)
\ge
\sum_{\lambda\in\operatorname{supp}\sigma(A)\cap
\operatorname{supp}\sigma(B)}
m_A(\lambda)m_B(\lambda).
\]
Consequently,
\[
\gamma_{\mathrm{gr}}^t(G\square F)
\le mn-
\sum_{\lambda\in\operatorname{supp}\sigma(A)\cap
\operatorname{supp}\sigma(B)}
m_A(\lambda)m_B(\lambda).
\]
\end{corollary}

\begin{proof}
Set $K=A\otimes I_n-I_m\otimes B$.  Lemma~\ref{lem:kronecker-pattern}
places $K$ in $\mathcal S_-(G\square F)$ in the first case and in
$\mathcal S_0(G\square F)$ in the second.  Both factor matrices are
diagonalizable over $\mathbb C$, so Lemma~\ref{lem:kronecker-nullity}
applies.  Since $K$ is real, its real and complex nullities agree.  The
claim now follows from \eqref{eq:matrix-skew-bounds} and
\eqref{eq:grundy-skew-identity}.
\end{proof}

\subsection{Path spectra}

The following construction is the finite-support Jacobi inverse-spectrum
argument; compare \cite{deBoorGolubJacobi}.  Its existence direction also
follows from \cite[Theorem~2.4]{HassaniMonfaredMallikMatchings}, since $P_n$
is connected and has matching number $\lfloor n/2\rfloor$.  We include the
Jacobi construction to make the tridiagonal form and the positivity of its
edge weights explicit.

\begin{lemma}\label{lem:path-skew-spectrum}
Let $n\ge2$ and $r=\lfloor n/2\rfloor$.  A multiset $\Lambda$ of $n$
complex numbers is the spectrum of a matrix in $\mathcal S_-(P_n)$ if and
only if there exist distinct positive real numbers
$0<\lambda_1<\cdots<\lambda_r$ such that
\[
\Lambda=
\begin{cases}
\{\pm i\lambda_1,\ldots,\pm i\lambda_r\},&n=2r,\\
\{0,\pm i\lambda_1,\ldots,\pm i\lambda_r\},&n=2r+1.
\end{cases}
\]
Moreover, the realizing matrix may be chosen in the form
\[
A=
\begin{pmatrix}
0&b_1&&&0\\
-b_1&0&b_2&&\\
&-b_2&0&\ddots&\\
&&\ddots&\ddots&b_{n-1}\\
0&&&-b_{n-1}&0
\end{pmatrix},
\qquad b_1,\ldots,b_{n-1}>0.
\]
\end{lemma}

\begin{proof}
Let $A\in\mathcal S_-(P_n)$, and denote its nonzero superdiagonal entries
by $c_1,\ldots,c_{n-1}$.  Diagonal switching makes them all positive:
take $\varepsilon_1=1$ and
\[
\varepsilon_{j+1}=\varepsilon_j\operatorname{sgn}(c_j),
\qquad E=\operatorname{diag}(\varepsilon_1,\ldots,\varepsilon_n).
\]
Then $A_+=EAE$ has superdiagonal entries $b_j=|c_j|$.

Let $J$ be the real symmetric tridiagonal matrix with zero diagonal and
positive sub- and superdiagonal entries $b_1,\ldots,b_{n-1}$, and put
$D=\operatorname{diag}(1,i,i^2,\ldots,i^{n-1})$.  Direct calculation gives
$D^*JD=iA_+$.  Every eigenvalue of $J$ is simple: the first coordinate of
an eigenvector determines all later coordinates through the tridiagonal
recurrence, and an eigenvector whose first coordinate is zero is zero.
Moreover, for $R=\operatorname{diag}(1,-1,1,-1,\ldots)$, we have
$RJR=-J$.  Hence the spectrum is simple and symmetric about zero, with
zero occurring exactly when $n$ is odd.  This proves necessity.

Conversely, fix the prescribed positive numbers and let
\[
\Omega=
\begin{cases}
\{-\lambda_r,\ldots,-\lambda_1,\lambda_1,\ldots,\lambda_r\},&n=2r,\\
\{-\lambda_r,\ldots,-\lambda_1,0,\lambda_1,\ldots,\lambda_r\},&n=2r+1.
\end{cases}
\]
Choose positive weights $w_t$ on $\Omega$ satisfying $w_t=w_{-t}$, and
define
\[
\langle f,g\rangle=\sum_{t\in\Omega}w_t f(t)g(t).
\]
Applying Gram--Schmidt to $1,x,\ldots,x^{n-1}$ gives orthonormal
polynomials $p_0,\ldots,p_{n-1}$ with positive leading coefficients.
Let $X$ be multiplication by $x$ and let $J$ be its matrix in this basis.
Self-adjointness and degree orthogonality show that $J$ is tridiagonal.
If $\kappa_k$ is the leading coefficient of $p_k$, then
\[
\langle Xp_k,p_{k+1}\rangle=\frac{\kappa_k}{\kappa_{k+1}}>0.
\]
The symmetric support and weights imply $p_k(-x)=(-1)^kp_k(x)$, and
therefore $\langle Xp_k,p_k\rangle=0$.  Thus $J$ has zero diagonal and
positive off-diagonal entries.  In the point-mass basis, $X$ is diagonal
with spectrum $\Omega$, so $\sigma(J)=\Omega$.  Form the skew-symmetric
matrix $A$ from the off-diagonal entries of $J$.  Since $D^*JD=iA$,
\[
 \sigma(A)=-i\sigma(J)=-i\Omega=\Lambda.
\]
Thus $A$ has exactly the prescribed spectrum.
\end{proof}

\begin{corollary}\label{cor:path-spectral-matching}
Let $2\le a\le b$, and define
\[
s(a,b)=
\begin{cases}
a-1,&a\text{ is odd and }b\text{ is even},\\
a,&\text{otherwise}.
\end{cases}
\]
Then matrices $A\in\mathcal S_-(P_a)$ and
$B\in\mathcal S_-(P_b)$ can be chosen so that
\[
\operatorname{nullity}(A\otimes I_b-I_a\otimes B)=s(a,b).
\]
Consequently,
\[
Z_-(P_a\square P_b)\ge s(a,b).
\]
\end{corollary}

\begin{proof}
Every path eigenvalue in Lemma~\ref{lem:path-skew-spectrum} is simple.
Choose $\lfloor a/2\rfloor$ distinct positive numbers for the nonzero
spectrum of the first factor and include the same numbers among the
positive spectral parameters of the second factor.  Choose all remaining
parameters disjointly.  The zero eigenvalue is shared exactly when both
$a$ and $b$ are odd.  Hence the spectra have exactly $s(a,b)$ common
eigenvalues.  Lemma~\ref{lem:kronecker-nullity} and
Corollary~\ref{cor:product-spectral-certificate} give the assertions.
\end{proof}

\begin{lemma}\label{lem:hollow-path-spectrum}
Let $n\ge2$ and $r=\lfloor n/2\rfloor$.  A multiset $\Lambda$ of $n$
real numbers is the spectrum of a matrix in $\mathcal S_0(P_n)$ if and
only if there exist distinct positive real numbers
$0<\lambda_1<\cdots<\lambda_r$ such that
\[
\Lambda=
\begin{cases}
\{\pm\lambda_1,\ldots,\pm\lambda_r\},&n=2r,\\
\{0,\pm\lambda_1,\ldots,\pm\lambda_r\},&n=2r+1.
\end{cases}
\]
Moreover, the realizing matrix may be chosen with positive sub- and
superdiagonal entries.
\end{lemma}

\begin{proof}
Every matrix in $\mathcal S_0(P_n)$ is an irreducible real symmetric
tridiagonal matrix.  The same recurrence and switching argument used in
Lemma~\ref{lem:path-skew-spectrum} show that its spectrum is simple and
symmetric about zero, with zero occurring exactly when $n$ is odd.  This
proves necessity.  Conversely, realize the corresponding purely imaginary
spectrum by a matrix $A\in\mathcal S_-(P_n)$ using
Lemma~\ref{lem:path-skew-spectrum}.  If $J$ is the hollow symmetric
tridiagonal matrix with the same positive edge weights and
$D=\operatorname{diag}(1,i,\ldots,i^{n-1})$, then $D^*JD=iA$.
Therefore $J$ has the prescribed real spectrum.
\end{proof}
\subsection{Signed cycles and path--cycle certificates}

Every eigenvalue of a real skew-symmetric matrix has the form $i\lambda$
with $\lambda\in\mathbb R$.  For a cyclic recurrence, Fourier modes have the
form $z^j$ with $z^n=\pm1$; the skew expression $z-z^{-1}$ produces
$2i\sin\theta$, whereas the symmetric expression $z+z^{-1}$ produces
$2\cos\theta$.  Thus the sine and cosine formulas below are the real
frequency forms of the same cyclic recurrence, and spectral matching amounts
to matching these frequencies with their multiplicities.

The switching classes and skew spectra of oriented graphs are standard;
see \cite{ShaderSoSkewSpectra}.  For $n\ge3$ and
$\varepsilon\in\{0,1\}$, define $S_n^\varepsilon$ by
\[
(S_n^\varepsilon)_{j,j+1}=1,
\qquad
(S_n^\varepsilon)_{j+1,j}=-1
\qquad (1\le j<n),
\]
and
\[
(S_n^\varepsilon)_{n,1}=(-1)^\varepsilon,
\qquad
(S_n^\varepsilon)_{1,n}=-(-1)^\varepsilon,
\]
with all other entries zero.

\begin{lemma}\label{lem:signed-cycle-spectrum}
Let $n\ge3$ and $\varepsilon\in\{0,1\}$.  Then, as a multiset,
\[
\sigma(S_n^\varepsilon)
=\left\{
2i\sin\left(\frac{(2k+\varepsilon)\pi}{n}\right):0\le k<n
\right\}.
\]
If $n$ is odd, all eigenvalues are simple, zero occurs exactly once, and
$S_n^0$ and $S_n^1$ are cospectral.  If $n$ is even, zero has multiplicity
two for $\varepsilon=0$ and does not occur for $\varepsilon=1$; every
eigenvalue other than $\pm2i$ that occurs has multiplicity two.  The
eigenvalues $2i$ and $-2i$ occur, each with multiplicity one, if and only if
\[
\frac n2\equiv\varepsilon\pmod2.
\]
\end{lemma}

\begin{proof}
For $0\le k<n$, put
\[
\theta_k=\frac{(2k+\varepsilon)\pi}{n},
\qquad z_k=e^{i\theta_k},
\qquad v_k=(1,z_k,\ldots,z_k^{n-1})^T.
\]
Since $z_k^n=(-1)^\varepsilon$, direct calculation, including at the two
wrap-around coordinates, gives
\[
S_n^\varepsilon v_k=(z_k-z_k^{-1})v_k
=2i\sin(\theta_k)v_k.
\]
The numbers $z_k$ are the distinct roots of
$z^n=(-1)^\varepsilon$, so the vectors $v_k$ form a Vandermonde basis.

Now $\sin\theta_k=\sin\theta_\ell$ if and only if either $k=\ell$ or
\[
2(k+\ell+\varepsilon)\equiv n\pmod{2n}.
\]
For odd $n$ the latter congruence has no solution, and zero occurs once.
The substitution $\ell\equiv(n-1)/2-k\pmod n$ shows that the two sign
choices are cospectral.  For even $n$, the second congruence pairs indices
under
\[
\ell\equiv\frac n2-\varepsilon-k\pmod n.
\]
Its fixed points exist exactly when $n/2\equiv\varepsilon\pmod2$ and
produce the simple values $\pm2i$.  Finally,
$2k+\varepsilon\equiv0\pmod n$ has two solutions for
$\varepsilon=0$ and none for $\varepsilon=1$.  The stated multiplicities
follow.
\end{proof}

\begin{corollary}\label{cor:path-signed-cycle-certificate}
Let $p\ge2$ and $c\ge3$, and define
\[
\eta(p,c)=
\begin{cases}
\min\{2p,c\},&c\text{ is even},\\
c-1,&c\text{ is odd},\ p\text{ is even},\text{ and }p\ge c,\\
\min\{p,c\},&\text{otherwise}.
\end{cases}
\]
Then
\[
\max_{\substack{A\in\mathcal S_-(P_p)\\
\varepsilon\in\{0,1\}}}
\operatorname{nullity}
\bigl(A\otimes I_c-I_p\otimes S_c^\varepsilon\bigr)
=\eta(p,c).
\]
Consequently, $Z_-(P_p\square C_c)\ge \eta(p,c)$.
\end{corollary}

\begin{proof}
Every eigenvalue of a matrix in $\mathcal S_-(P_p)$ is simple.  Suppose
first that $c$ is even.  Choose $\varepsilon=0$ when $p$ is odd and
$\varepsilon=1$ when $p$ is even.  If $2p\le c$, the multiplicity-two
eigenvalues of $S_c^\varepsilon$ contain a negation-invariant $p$-element
set, with zero included exactly when $p$ is odd.  Indeed, if the simple
values $\pm2i$ occur, then $c/2\equiv\varepsilon\pmod2$, whereas $p$ has
the opposite parity, so at least $p$ double eigenvalues remain.  Realize
the selected set as a path spectrum.  Every common eigenvalue contributes
two, giving nullity $2p$.

If $c<2p$, the set of distinct eigenvalues of
$S_c^\varepsilon$ has at most $p$ elements, is invariant under negation,
and contains zero exactly when $p$ is odd.  Extend it, if necessary, to an
admissible $p$-element path spectrum.  All eigenvalues of the cycle factor
are then common, so their multiplicities sum to $c$.  This proves the
even-cycle formula and its optimality within the family in the statement,
since a simple $p$-point path spectrum can contribute at most two per
eigenvalue and at most $c$ in total.

Now let $c$ be odd.  The cycle spectrum is simple and contains zero.  If
$p<c$, select a negation-invariant $p$-element subset and realize it as a
path spectrum, giving nullity $p$.  If $p\ge c$ and $p$ is odd, include
the entire cycle spectrum, giving nullity $c$.  If $p\ge c$ and $p$ is
even, zero cannot occur in the path spectrum, but all $c-1$ nonzero cycle
eigenvalues can occur.  For odd $c$, both factor spectra are simple, so at
most $\min\{p,c\}$ eigenvalues can be common; in the final case, excluding
zero reduces this bound to $c-1$.  Thus these values are optimal within the
same normalized signed-cycle family.
Corollary~\ref{cor:product-spectral-certificate} yields the lower bound
on $Z_-$.
\end{proof}

When $p$ is even and $c$ is odd, the product order $pc$ is even, and the
nullity of every real skew-symmetric matrix of this order is even.  Thus
no skew-symmetric matrix certificate can attain the odd target $c$ in the
exceptional range $p\ge c$.  This parity obstruction motivates the hollow
symmetric construction below.

\begin{lemma}\label{lem:odd-cycle-adjacency-spectrum}
Let $c\ge3$ be odd, and let $H_c=A(C_c)$ be the ordinary adjacency matrix
of $C_c$.  Then $H_c\in\mathcal S_0(C_c)$ and
\[
\sigma(H_c)=
\left\{2\cos\left(\frac{2\pi k}{c}\right):0\le k<c\right\}.
\]
Its distinct eigenvalues form the set
\[
T_c=\left\{2\cos\left(\frac{2\pi k}{c}\right):
0\le k\le\frac{c-1}{2}\right\}.
\]
The eigenvalue $2$ has multiplicity one, every element of
$T_c\setminus\{2\}$ has multiplicity two, and
$0\notin T_c$ and $T_c\cap(-T_c)=\varnothing$.
\end{lemma}

\begin{proof}
With $\omega=e^{2\pi i/c}$, the Fourier vector
$u_k=(1,\omega^k,\ldots,\omega^{(c-1)k})^T$ satisfies
\[
H_cu_k=(\omega^k+\omega^{-k})u_k
=2\cos\left(\frac{2\pi k}{c}\right)u_k.
\]
Equality of two cosine values occurs exactly for indices $k$ and $-k$
modulo $c$.  Since $c$ is odd, only $k=0$ is fixed, which gives the stated
multiplicities.  The equations producing zero or a pair of opposite cosine
values would equate an even integer with an odd integer after denominators
are cleared.  Hence $0\notin T_c$ and
$T_c\cap(-T_c)=\varnothing$.
\end{proof}

\begin{lemma}\label{lem:pc-exceptional-hollow-certificate}
Let $p\ge3$ be even and let $c\ge3$ be odd.  If $p\ge c$, then there
exist $J_p\in\mathcal S_0(P_p)$ and $H_c\in\mathcal S_0(C_c)$ such that
\[
\operatorname{nullity}
\bigl(J_p\otimes I_c-I_p\otimes H_c\bigr)=c.
\]
Consequently, $Z_-(P_p\square C_c)\ge c$.
\end{lemma}

\begin{proof}
Take $H_c=A(C_c)$ and let $T_c$ be as in
Lemma~\ref{lem:odd-cycle-adjacency-spectrum}.  Because $p$ is even, $c$ is
odd, and $p\ge c$, we have $p\ge c+1$.  The set
$T_c\cup(-T_c)$ consists of $c+1$ distinct nonzero real numbers and is
invariant under negation.  If $p>c+1$, adjoin further disjoint opposite
pairs.  Lemma~\ref{lem:hollow-path-spectrum} then provides
$J_p\in\mathcal S_0(P_p)$ with simple spectrum containing every element of
$T_c$.  Lemma~\ref{lem:kronecker-nullity} gives
\[
\operatorname{nullity}
\bigl(J_p\otimes I_c-I_p\otimes H_c\bigr)
=\sum_{\lambda\in T_c}m_{H_c}(\lambda)=c.
\]
The final assertion follows from Lemma~\ref{lem:kronecker-pattern} and
\eqref{eq:matrix-skew-bounds}.
\end{proof}

\begin{corollary}\label{cor:pc-complete-matrix-bound}
Let $p\ge2$ and $c\ge3$, and define
\[
q(p,c)=
\begin{cases}
\min\{p,c\},&c\text{ is odd},\\
\min\{2p,c\},&c\text{ is even}.
\end{cases}
\]
Then
\[
Z_-(P_p\square C_c)\ge q(p,c).
\]
\end{corollary}

\begin{proof}
If $p=2$, then the exceptional condition $p\ge c$ cannot occur, and
Corollary~\ref{cor:path-signed-cycle-certificate} gives the stated bound.
Assume henceforth that $p\ge3$.
Corollary~\ref{cor:path-signed-cycle-certificate} proves the bound unless
$c$ is odd, $p$ is even, and $p\ge c$.  In that remaining case,
$q(p,c)=c$ and Lemma~\ref{lem:pc-exceptional-hollow-certificate} applies.
\end{proof}

%% file: 04_path_path.tex
\section{Products of two paths}\label{sec:path-path}

A boundary column gives a forcing set of size $a$ for every path product.
The following construction improves this bound by one in the only
exceptional parity case.

\begin{lemma}\label{lem:path-product-checkerboard-forcing}
Let $3\le a\le b$, with $a$ odd and $b$ even.  Then
\[
Z_-(P_a\square P_b)\le a-1.
\]
\end{lemma}

\begin{proof}
Write
\[
V(P_a\square P_b)=
\{v_{i,j}:1\le i\le a,\ 1\le j\le b\},
\]
where $v_{i,j}$ is adjacent to $v_{i',j'}$ precisely when
$|i-i'|+|j-j'|=1$.  Initially color
\[
B=\{v_{i,1}:i\equiv0\pmod2\}
\cup\{v_{i,b}:i\equiv0\pmod2\}
\]
blue.  Since $a$ is odd, $|B|=a-1$.

Let
\[
X=\{v_{i,j}:i+j\text{ is odd}\},
\qquad
Y=\{v_{i,j}:i+j\text{ is even}\}.
\]
The initial vertices in the first column are exactly the vertices of $X$
there.  For $j=1,\ldots,b-1$ in this order, perform the following forces
in any order within stage $j$:
\begin{equation}\label{eq:path-product-left-sweep}
v_{i,j}\longrightarrow v_{i,j+1}
\qquad(i\equiv j\pmod2).
\end{equation}
Before stage $j$, all vertices of $X$ in columns $1,\ldots,j$ are blue.
A displayed source lies in $Y$; its vertical neighbors and its left
neighbor, when present, lie in $X$ and are blue, while its right neighbor
is its unique white neighbor.  Thus all displayed forces are legal and
color the vertices of $X$ in column $j+1$.  Induction shows that all of
$X$ becomes blue.

Because $b$ is even, the initial vertices in the last column are exactly
the vertices of $Y$ there.  For $j=b,b-1,\ldots,2$, perform
\begin{equation}\label{eq:path-product-right-sweep}
v_{i,j}\longrightarrow v_{i,j-1}
\qquad(i\equiv j-1\pmod2).
\end{equation}
Before stage $j$, the vertices of $Y$ in columns $j,\ldots,b$ are blue.
A displayed source lies in $X$; its vertical neighbors and right neighbor,
when present, lie in $Y$ and are blue, while its left neighbor is its
unique white neighbor.  Descending induction colors all of $Y$.  Hence
$B$ is a skew zero forcing set of cardinality $a-1$.
\end{proof}

\begin{proof}[Proof of Theorem~\ref{thm:path-product}]
Put
\[
s(a,b)=a-\mathbf 1_{\{a\text{ is odd and }b\text{ is even}\}}.
\]
Corollary~\ref{cor:path-spectral-matching} gives
\[
Z_-(P_a\square P_b)\ge s(a,b).
\]

For the reverse inequality, if $a$ is odd and $b$ is even,
Lemma~\ref{lem:path-product-checkerboard-forcing} gives a skew zero forcing
set of cardinality $a-1=s(a,b)$.  In every other case, color the first
column blue.  Once columns $1,\ldots,j$ are blue with $j<b$, each $v_{i,j}$ has
$v_{i,j+1}$ as its unique white neighbor.  Hence column $j$ forces column
$j+1$, and induction colors the entire graph from $a=s(a,b)$ initial
vertices.

The two inequalities prove the formula for $Z_-$.  The formula for
$\gamma_{\mathrm{gr}}^t$ follows from
\eqref{eq:grundy-skew-identity}.
\end{proof}

\begin{corollary}\label{cor:path-product-minimum-skew-rank}
Let $2\le a\le b$, and put
\[
s(a,b)=a-\mathbf 1_{\{a\text{ is odd and }b\text{ is even}\}}.
\]
Then
\[
M_-(P_a\square P_b)=s(a,b)
\]
and
\[
\operatorname{mr}^-(P_a\square P_b)=ab-s(a,b).
\]
\end{corollary}

\begin{proof}
The matrices constructed in
Corollary~\ref{cor:path-spectral-matching} give
$M_-(P_a\square P_b)\ge s(a,b)$.  On the other hand,
\eqref{eq:matrix-skew-bounds} and
Theorem~\ref{thm:path-product} give
\[
M_-(P_a\square P_b)\le Z_-(P_a\square P_b)=s(a,b).
\]
The minimum-rank formula follows from the definition of
$\operatorname{mr}^-$.
\end{proof}

%% file: 05_path_cycle.tex
\section{Products of a path and a cycle}\label{sec:path-cycle}

Use vertices
\[
v_{i,j},\qquad 0\le i\le p-1,\quad j\in\mathbb Z_c,
\]
where the first coordinate is the path coordinate, and put
\[
R_i=\{v_{i,j}:j\in\mathbb Z_c\},
\qquad
Q_j=\{v_{i,j}:0\le i\le p-1\}.
\]
Thus $R_i$ is a $C_c$-layer and $Q_j$ is a $P_p$-layer.

\begin{lemma}\label{lem:parallel-skew-round-serialization}
Suppose that, at the beginning of a stage, a finite family of ordered pairs
$(x_\alpha,y_\alpha)$ has the property that $y_\alpha$ is the unique white
neighbor of $x_\alpha$ for every $\alpha$.  Then all vertices among the
$y_\alpha$ can be colored blue by a sequential skew zero forcing process.
\end{lemma}

\begin{proof}
Order the pairs arbitrarily.  When $(x_\alpha,y_\alpha)$ is reached, omit
it if $y_\alpha$ is already blue.  Otherwise, earlier forces can only have
decreased the white-neighbor set of $x_\alpha$, so $y_\alpha$ remains its
unique white neighbor and $x_\alpha\to y_\alpha$ is legal.
\end{proof}

\begin{lemma}\label{lem:pc-end-cycle-layer-sweep}
One boundary cycle layer is a skew zero forcing set of
$P_p\square C_c$.  In particular,
\[
Z_-(P_p\square C_c)\le c.
\]
\end{lemma}

\begin{proof}
Color $R_0$ blue.  Every vertex of $R_0$ has its two cycle-direction
neighbors blue and has exactly one path-direction neighbor, namely the
corresponding vertex of $R_1$.  Thus $R_0$ colors $R_1$.  Once
$R_0,\ldots,R_i$ are blue, every vertex of $R_i$ has all its neighbors
blue except its corresponding neighbor in $R_{i+1}$, so $R_i$ colors
$R_{i+1}$.  Lemma~\ref{lem:parallel-skew-round-serialization} serializes
each layer, and induction colors all cycle layers.
\end{proof}

\begin{lemma}\label{lem:pc-two-layer-propagation}
If two adjacent path layers $Q_j\cup Q_{j+1}$ are blue, then all of
$P_p\square C_c$ can be colored blue.
\end{lemma}

\begin{proof}
Every vertex of $Q_j$ has its path-direction neighbors and its neighbor in
$Q_{j+1}$ blue, so it can force its neighbor in $Q_{j-1}$ whenever that
neighbor is still white.  Similarly, $Q_{j+1}$ colors $Q_{j+2}$.  The
rounds can be serialized by
Lemma~\ref{lem:parallel-skew-round-serialization}.  Repeating this expansion
in both cyclic directions colors every path layer blue.
\end{proof}

Lemma~\ref{lem:pc-end-cycle-layer-sweep} gives the elementary bound $c$.
When $c$ is odd and $p<c$, however, the target value is $p<c$, so a
genuinely smaller forcing set is required.  The next lemma supplies exactly
this improvement.

\begin{lemma}\label{lem:pc-odd-step-two-forcing}
Let $2\le p<c$, where $c=2m+1$ is odd.  Then
\[
B_{p,c}=\{v_{0,\overline{2r}}:0\le r\le p-1\}
\]
is a skew zero forcing set of $P_p\square C_c$.  In particular,
\[
Z_-(P_p\square C_c)\le p.
\]
\end{lemma}

\begin{proof}
All second coordinates below are reduced modulo $c$; an overline is used
when the reduction needs emphasis.  Since $c$ is odd, multiplication by
$2$ is a permutation of $\mathbb Z_c$.  Because $p<c$, the set $B_{p,c}$
has cardinality $p$.

First suppose that $(p,c)=(2,3)$.  Starting from
$B_{2,3}=\{v_{0,0},v_{0,2}\}$, perform
\[
 v_{0,1}\longrightarrow v_{1,1},\qquad
 v_{1,0}\longrightarrow v_{1,2},\qquad
 v_{1,2}\longrightarrow v_{1,0},\qquad
 v_{0,2}\longrightarrow v_{0,1}.
\]
At each step the displayed target is the unique white neighbor of the
source; recall that the skew color-change rule permits either a blue or a
white source.  Hence the assertion holds in this case.  Assume henceforth
that $c\ge5$.

Put $D_0=B_{p,c}$.  For $1\le t\le p-1$, define
\[
D_t=\{v_{t,\overline{t+2r}}:0\le r\le p-t-1\}.
\]
We first color $D_1,D_2,\ldots,D_{p-1}$ in this order.  If
$v_{t,\bar j}\in D_t$, use $v_{t-1,\bar j}$ as the source.  Its two
cycle-direction neighbors belong to $D_{t-1}$.  When $t\ge2$, its remaining
path-direction neighbor belongs to $D_{t-2}$; when $t=1$, that neighbor is
absent.  Hence the displayed target is the unique white neighbor of its
source at the beginning of stage $t$.  The targets are distinct, and
Lemma~\ref{lem:parallel-skew-round-serialization} serializes the stage.

For $p\le t\le T:=p+m-1$, set
\[
h_t=2p-1-t,\qquad
\ell_t=h_t-2,\qquad
r_t=t+1,\qquad
\tau_t=\max\{0,h_t\},
\]
and
\[
I_t=\{i\in\{0,\ldots,p-1\}:i\ge\tau_t,\
i\equiv h_t\pmod2\}.
\]
In round $t$, color
\begin{equation}\label{eq:pc-reflected-wave}
E_t=\{v_{i,\overline{\ell_t}},v_{i,\overline{r_t}}:i\in I_t\}.
\end{equation}
The two residues in \eqref{eq:pc-reflected-wave} are distinct.  Indeed,
\[
r_t-\ell_t=2(t-p+2),\qquad 2\le t-p+2\le m+1.
\]
Thus $4\le r_t-\ell_t\le c+1$, and the only possible positive multiple
of $c$ in this interval is $c$, which is odd whereas the displayed
difference is even.

Figure~\ref{fig:pc-four-nine-forcing} illustrates the two parts of the
construction for $P_4\square C_9$.  The $D_t$-vertices form the initial
diagonal front, while the $E_t$-vertices expand toward the cyclic seam.

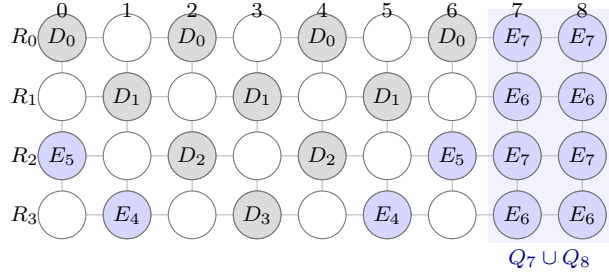
\begin{figure}[htbp]
\centering
\input{fig_pc_4_9}
\caption{The forcing stages for $P_4\square C_9$.  Unlabeled vertices
are still white at the end of round $E_7$; the adjacent blue layers
$Q_7\cup Q_8$ then color the remainder by
Lemma~\ref{lem:pc-two-layer-propagation}.  Cycle edges joining columns
$8$ and $0$ are omitted from the drawing.}
\label{fig:pc-four-nine-forcing}
\end{figure}

For $t=p$, we have $I_p=\{p-1\}$.  The sources
$v_{p-1,p-2}$ and $v_{p-1,p}$ force $v_{p-1,p-3}$ and
$v_{p-1,p+1}$, respectively: their inward cycle-direction neighbors lie
in $D_{p-1}$ and their upper neighbors lie in $D_{p-2}$.

Suppose now that $t>p$ and that all earlier rounds have been performed.
Fix $i\in I_t$ and write $h=h_t$.  For the left-hand target use
\[
x=v_{i,\overline{h-1}}
\quad\text{to force}\quad
v_{i,\overline{h-2}}.
\]
The inward neighbor $v_{i,\bar h}$ is already blue: if $i=h\ge0$, it is
the left endpoint of $D_i$; otherwise it belongs to $E_{t-2}$.  Every
existing path-direction neighbor of $x$ is also blue.  The row $i+1$, when
present, belongs to $I_{t-1}$, and the row $i-1$, when present, either
belongs to $I_{t-1}$ or supplies the left endpoint of $D_{i-1}$ in the
leading case.  Since $\ell_{t-1}=h-1$, these are precisely previously
colored vertices.  Thus the left-hand target, if white, is the unique white
neighbor of $x$.

For the right-hand target use
\[
y=v_{i,\bar t}
\quad\text{to force}\quad
v_{i,\overline{t+1}}.
\]
Its inward neighbor is the right endpoint of $D_i$ when $i=h$ and otherwise
belongs to $E_{t-2}$ because $r_{t-2}=t-1$.  Its existing
path-direction neighbors were colored in $E_{t-1}$, except that in the
leading case the upper one is the right endpoint of $D_{i-1}$; here
$r_{t-1}=t$.  Hence this force is legal whenever its target is white.
Lemma~\ref{lem:parallel-skew-round-serialization} serializes the round,
with a force omitted if its target became blue earlier in that round.

It remains to check the cyclic seam.  Put $j_*=p+m-1=T$.  At the final
round,
\[
\ell_T=p-m-2\equiv j_*\pmod c,
\qquad r_T=p+m=j_*+1.
\]
Moreover, $E_{T-1}$ uses the same two residues in the opposite order.
If $p\le m$, then $I_T$ and $I_{T-1}$ are the complementary parity
classes of all rows, so $Q_{\overline{j_*}}$ and
$Q_{\overline{j_*+1}}$ are blue.

Assume $p>m$ and put $s=p-m$.  The sets $I_T$ and $I_{T-1}$ cover all
rows $i\ge s$.  For $i<s$, the two seam vertices already lie in $D_i$.
If $i\equiv s\pmod2$, then $i\le s-2$ and $D_i$ contains the integer
columns $s-2=j_*-c$ and $p+m=j_*+1$.  If
$i\not\equiv s\pmod2$, then $i\le s-1$ and $D_i$ contains
$s-1=j_*+1-c$ and $p+m-1=j_*$.
These integers have the required parity and lie between the endpoints
$i$ and $2p-2-i$ of $D_i$.  Thus the adjacent layers
$Q_{\overline{j_*}}$ and $Q_{\overline{j_*+1}}$ are blue in every case.
Lemma~\ref{lem:pc-two-layer-propagation} now colors the rest of the graph.
\end{proof}

\begin{proof}[Proof of Theorem~\ref{thm:pc-exact-skew-zero-forcing}]
Corollary~\ref{cor:pc-complete-matrix-bound} gives
\[
Z_-(P_p\square C_c)\ge q(p,c).
\]
For the reverse inequality,
Lemma~\ref{lem:pc-end-cycle-layer-sweep} gives the bound $c$, while
Lemma~\ref{lem:pc-two-layer-propagation}, applied with two adjacent
path layers initially blue, gives the bound $2p$.  Hence
\[
Z_-(P_p\square C_c)\le\min\{2p,c\}.
\]
This is $q(p,c)$ when $c$ is even, and the bound $c=q(p,c)$ applies when
$c$ is odd and $p\ge c$.  In the remaining case, $c$ is odd and $p<c$;
Lemma~\ref{lem:pc-odd-step-two-forcing} gives
$Z_-(P_p\square C_c)\le p=q(p,c)$.  Equality follows in all cases, and
the formula for $\gamma_{\mathrm{gr}}^t$ follows from
\eqref{eq:grundy-skew-identity}.
\end{proof}

\begin{corollary}\label{cor:pc-minimum-skew-rank}
Let $p\ge2$ and $c\ge3$, and define
\[
\eta(p,c)=
\begin{cases}
\min\{2p,c\},&c\text{ is even},\\
c-1,&c\text{ is odd},\ p\text{ is even},\text{ and }p\ge c,\\
\min\{p,c\},&\text{otherwise}.
\end{cases}
\]
Then
\[
M_-(P_p\square C_c)=\eta(p,c)
\]
and
\[
\operatorname{mr}^-(P_p\square C_c)=pc-\eta(p,c).
\]
In the exceptional branch where $c$ is odd, $p$ is even, and $p\ge c$,
\[
M_-(P_p\square C_c)=c-1<c=Z_-(P_p\square C_c)
\]
and
\[
M_0(P_p\square C_c)=c.
\]
\end{corollary}

\begin{proof}
Corollary~\ref{cor:path-signed-cycle-certificate} constructs a matrix in
$\mathcal S_-(P_p\square C_c)$ of nullity $\eta(p,c)$, so
$M_-(P_p\square C_c)\ge \eta(p,c)$.  Outside the exceptional branch,
$\eta(p,c)=q(p,c)$, and Theorem~\ref{thm:pc-exact-skew-zero-forcing} together
with \eqref{eq:matrix-skew-bounds} gives the reverse inequality.

In the exceptional branch, $pc$ is even, so every real skew-symmetric
matrix of order $pc$ has even nullity.  Since
$Z_-(P_p\square C_c)=c$ is odd, \eqref{eq:matrix-skew-bounds} gives
$M_-(P_p\square C_c)\le c-1=\eta(p,c)$.  Finally,
Lemma~\ref{lem:pc-exceptional-hollow-certificate} gives $M_0\ge c$, while
\eqref{eq:matrix-skew-bounds} gives $M_0\le Z_-=c$.  The formula for
$\operatorname{mr}^-$ follows from its definition.
\end{proof}

%% file: fig_pc_4_9.tex
\begin{tikzpicture}[
  x=0.86cm,
  y=0.78cm,
  vertex/.style={circle,draw=black!55,fill=white,minimum size=6.3mm,
    inner sep=0pt},
  dvertex/.style={vertex,fill=black!14,font=\scriptsize},
  evertex/.style={vertex,fill=blue!16,font=\scriptsize}
]
  \fill[blue!5] (6.55,0.48) rectangle (8.45,-3.48);

  \foreach \i in {0,...,3}{
    \foreach \j in {0,...,7}{
      \draw[black!28] (\j,-\i)--(\j+1,-\i);
    }
  }
  \foreach \i in {0,...,2}{
    \foreach \j in {0,...,8}{
      \draw[black!28] (\j,-\i)--(\j,-\i-1);
    }
  }

  \foreach \i in {0,...,3}{
    \foreach \j in {0,...,8}{
      \node[vertex] at (\j,-\i) {};
    }
  }

  \foreach \j in {0,2,4,6}{
    \node[dvertex] at (\j,0) {$D_0$};
  }
  \foreach \j in {1,3,5}{
    \node[dvertex] at (\j,-1) {$D_1$};
  }
  \foreach \j in {2,4}{
    \node[dvertex] at (\j,-2) {$D_2$};
  }
  \node[dvertex] at (3,-3) {$D_3$};

  \foreach \j in {1,5}{
    \node[evertex] at (\j,-3) {$E_4$};
  }
  \foreach \j in {0,6}{
    \node[evertex] at (\j,-2) {$E_5$};
  }
  \foreach \i in {1,3}{
    \foreach \j in {7,8}{
      \node[evertex] at (\j,-\i) {$E_6$};
    }
  }
  \foreach \i in {0,2}{
    \foreach \j in {7,8}{
      \node[evertex] at (\j,-\i) {$E_7$};
    }
  }

  \foreach \j in {0,...,8}{
    \node[font=\scriptsize,above=4pt] at (\j,0) {$\j$};
  }
  \foreach \i in {0,...,3}{
    \node[font=\scriptsize,left=5pt] at (0,-\i) {$R_{\i}$};
  }
  \node[font=\scriptsize,blue!55!black] at (7.5,-3.75)
    {$Q_7\cup Q_8$};
\end{tikzpicture}

%% file: 06_cycle_cycle.tex
\section{Products of two cycles}\label{sec:cycle-cycle}

Throughout this section, let $3\le a\le b$.  Theorems~5.1(i)
and~5.2(iii) of Bre\v{s}ar et al.~\cite{BresarEtAlProductGraphs},
together with \eqref{eq:grundy-skew-identity}, give
\begin{equation}\label{eq:known-cc-interval}
 a\le Z_-(C_a\square C_b)\le2a.
\end{equation}
The upper bound comes from two adjacent $C_a$-columns.  We cite that
standard construction and give explicit forcing sets only when fewer than
$2a$ initial vertices are needed.

\subsection{Forcing sets below two columns}

\begin{lemma}\label{lem:odd-even-half-layer-forcing}
Let $o\ge3$ be odd and let $e\ge4$ be even.  Then
\[
 Z_-(C_o\square C_e)\le e.
\]
\end{lemma}

\begin{proof}
Index the vertices as
\[
 v_{i,j},\qquad i\in\mathbb Z_o,\quad j\in\mathbb Z_e.
\]
For every integer $t$, let $\overline t$ denote its residue modulo $o$ and
define
\[
 X_t=\{v_{\overline t,j}:t+j\equiv0\pmod2\}.
\]
Each $X_t$ has $e/2$ vertices.  Initially color
$B=X_0\cup X_1$, so $|B|=e$.

We claim successively that $X_2,X_3,\ldots,X_{2o-1}$ can be colored.
Suppose that $1\le t\le2o-2$ and that $X_0,\ldots,X_t$ are blue.  For
each $j$ satisfying $t+j\equiv1\pmod2$, use
$v_{\overline t,j}$ as a forcing vertex.  Its two $C_e$-neighbors lie in
$X_t$, its neighbor $v_{\overline{t-1},j}$ lies in $X_{t-1}$, and its
remaining neighbor $v_{\overline{t+1},j}$ lies in $X_{t+1}$.
The sets $X_0,\ldots,X_{2o-1}$ are pairwise disjoint: two can occupy the
same row only when their indices differ by $o$, and this odd shift
interchanges the two column parities.  Hence
\[
 v_{\overline t,j}\longrightarrow v_{\overline{t+1},j}.
\]
The forces for fixed $t$ have distinct targets and may be serialized by
Lemma~\ref{lem:parallel-skew-round-serialization}.

For each row, the two sets among $X_0,\ldots,X_{2o-1}$ supported on that
row have complementary column parities.  These $2o$ half-layers therefore
partition $V(C_o\square C_e)$, so $B$ is a skew zero forcing set.
\end{proof}

\begin{lemma}\label{lem:odd-odd-closed-diagonal-forcing}
Let $3\le a<b$ be odd.  Then
\[
 Z_-(C_a\square C_b)\le b.
\]
\end{lemma}

\begin{proof}
Write $b=a+2s$, where $s\ge1$, and index the vertices as
\[
 v_{i,j},\qquad i\in\mathbb Z_a,\quad j\in\mathbb Z_b.
\]
Let $D$ increase the first coordinate by one and let $U$ decrease it by
one.  The cyclic word
\[
 W=D^{a+s}U^s
\]
is read from column $j$ to column $j+1$: the letter $D$ means
$h_{j+1}=h_j+1$, and the letter $U$ means $h_{j+1}=h_j-1$ in
$\mathbb Z_a$.  Its net displacement is $a$, which is zero modulo $a$.
It therefore determines a closed height profile
$h=(h_0,\ldots,h_{b-1})\in\mathbb Z_a^b$.  Fix $h_0=0$ and initially
color
\[
 B=\{v_{h_j,j}:j\in\mathbb Z_b\}.
\]

We use one local move.  Suppose that the profile contains an adjacent
pair $DU$, with
\[
 h_{j-1}=r,\qquad h_j=r+1,\qquad h_{j+1}=r.
\]
The vertex $v_{r,j}$ has the three blue neighbors
$v_{r,j-1}$, $v_{r+1,j}$, and $v_{r,j+1}$.  Thus it forces
$v_{r-1,j}$ if that vertex is white; if it is already blue, no force is
needed.  In either case, the blue closure contains the profile obtained
by replacing $DU$ with $UD$.

Move the leftmost $U$ across the initial block of $D$'s, and then move
the last $D$ of the resulting $D$-block across the remaining $U$'s:
\[
 D^{a+s}U^s
 \ \rightsquigarrow\
 U D^{a+s}U^{s-1}
 \ \rightsquigarrow\
 U D^{a+s-1}U^{s-1}D.
\]
The last $D$ and the first $U$ now form a $DU$ pair across the column
seam.  Switching that pair restores the word $W$ and lowers the height
in column $0$ by two.  Hence every height is lowered by two, and the blue
closure contains
\[
 B-2=\{v_{h_j-2,j}:j\in\mathbb Z_b\}.
\]
Repeating the same switches gives $B-2k$ for every $k\ge0$.  Since $a$
is odd, the residues $h_j-2k$, $0\le k<a$, run through
$\mathbb Z_a$ in every fixed column $j$.  Their union is the whole
graph, so $B$ is a skew zero forcing set of size $b$.
\end{proof}

\begin{corollary}\label{cor:cycle-forcing-upper-bounds}
Let $3\le a<b$.  If $a$ is odd, then
\[
 Z_-(C_a\square C_b)\le\min\{b,2a\}.
\]
If $a$ is even and $b$ is odd, then
\[
 Z_-(C_a\square C_b)\le a.
\]
\end{corollary}

\begin{proof}
Suppose first that $a$ is odd.  If $b$ is odd, use
Lemma~\ref{lem:odd-odd-closed-diagonal-forcing}; if $b$ is even, use
Lemma~\ref{lem:odd-even-half-layer-forcing} with $(o,e)=(a,b)$.
Combine the resulting bound $b$ with the known bound $2a$ in
\eqref{eq:known-cc-interval}.  If $a$ is even and $b$ is odd, use the
natural isomorphism $C_a\square C_b\cong C_b\square C_a$ and apply the
half-layer lemma with $(o,e)=(b,a)$.
\end{proof}

\subsection{A cyclic column-defect lemma}

For a finite legal prefix $\mathcal S=(v_1,\ldots,v_r)$, let
$S=\{v_1,\ldots,v_r\}$ and put
\[
 \delta(S)=|N(S)|-|S|,
 \qquad N(S)=\bigcup_{v\in S}N_G(v).
\]
If $F_i=N(v_i)\setminus\bigcup_{h<i}N(v_h)$ is the set footprinted at
step $i$, then
\begin{equation}\label{eq:defect-footprint-excess}
 \delta(S)=\sum_{i=1}^r(|F_i|-1).
\end{equation}
Thus the defect is exactly the cumulative excess beyond the one new
neighbor required for legality.  This elementary bookkeeping is implicit
in standard footprint arguments~\cite{BresarHenningRallTotalSequences};
the next lemma packages it columnwise and may be viewed as a refinement of
the arc-counting method used for Cartesian products in
\cite[Theorem~5.1]{BresarEtAlProductGraphs}.

\begin{lemma}\label{lem:cyclic-column-defect}
Let $q\ge2$ and $b\ge3$.  Suppose that integers $t_j,u_j$, indexed by
$j\in\mathbb Z_b$, satisfy
\[
 t_0=0,\qquad 1\le t_j\le q\quad(j\ne0),
\]
and
\begin{equation}\label{eq:cyclic-column-local-bound}
 0\le u_j\le
 \min\{t_{j-1},t_{j+1},t_j-\mathbf1_{\{0<t_j<q\}}\}.
\end{equation}
Then
\begin{equation}\label{eq:cyclic-column-defect-bound}
 \sum_{j\in\mathbb Z_b}(t_j-u_j)
 \ge\min\{b-1,2q\}.
\end{equation}
Moreover, if $b\le2q$ and
$\sum_j(t_j-u_j)=b-1$, then
\[
 t_1=t_{b-1}=1.
\]
\end{lemma}

\begin{proof}
Put $\varepsilon_j=\mathbf1_{\{0<t_j<q\}}$ and
\[
 c_j=t_j-\min\{t_{j-1},t_{j+1},t_j-\varepsilon_j\}.
\]
By \eqref{eq:cyclic-column-local-bound},
\begin{equation}\label{eq:cyclic-column-charge-sum}
 \sum_j(t_j-u_j)\ge\sum_jc_j.
\end{equation}
Cut the cyclic height sequence at $t_0=0$ and write $t_b=t_0$.
If height $q$ occurs at most once among $t_1,\ldots,t_{b-1}$, then
$c_j\ge1$ for every $1\le j\le b-1$: this follows from
$\varepsilon_j=1$ below height $q$, while a unique term of height $q$
has a neighbor of smaller height.  Hence $\sum_jc_j\ge b-1$.

If height $q$ occurs at least twice, let $p<s$ be its first and last
occurrences.  Since
\[
 c_j\ge\max\{0,t_j-t_{j-1},t_j-t_{j+1}\},
\]
the total charge on $1,\ldots,p$ is at least $q$, and that on
$s,\ldots,b-1$ is also at least $q$.  The two intervals are disjoint,
so $\sum_jc_j\ge2q$.  This proves
\eqref{eq:cyclic-column-defect-bound}.

Finally, assume that $b\le2q$ and $\sum_j(t_j-u_j)=b-1$.  Two terms of
height $q$ would give $\sum_jc_j\ge2q>b-1$, so height $q$ occurs at most
once.  Consequently every nonzero-column charge is at least one, and
equality in \eqref{eq:cyclic-column-charge-sum} gives
$c_1=c_{b-1}=1$.  Because $t_0=0$, we have
$c_1=t_1$ and $c_{b-1}=t_{b-1}$, as required.
\end{proof}

\begin{proposition}\label{prop:odd-shorter-cycle-lower-bound}
Let $3\le a<b$, with $a$ odd, and put $\tau=\min\{b,2a\}$.  Then
\[
 \gamma_{\mathrm{gr}}^t(C_a\square C_b)\le ab-\tau,
 \qquad
 Z_-(C_a\square C_b)\ge \tau.
\]
\end{proposition}

\begin{proof}
Put $n=ab$ and suppose that a Grundy total dominating sequence
$\mathcal S=(v_1,\ldots,v_k)$ has length $k\ge n-\tau+1$.  Write
$S_r=\{v_1,\ldots,v_r\}$ and
\[
 D_r=|N(S_r)|-|S_r|.
\]
If $F_r=N(v_r)\setminus N(S_{r-1})$ is the footprint of $v_r$, then
\[
 D_r-D_{r-1}=|F_r|-1\ge0.
\]
Thus $(D_r)$ is nondecreasing.  Since $\mathcal S$ has maximum length and
$G$ has no isolated vertices, $N(S_k)=V(G)$, and hence
\begin{equation}\label{eq:terminal-cycle-defect}
 D_k=n-k\le \tau-1.
\end{equation}

Some $C_a$-column is contained in $S_k$.  Otherwise
$k\le(a-1)b=n-b$, contradicting $k\ge n-\tau+1$ when $\tau=b$; when
$\tau=2a<b$, it contradicts $n-b<n-2a+1\le k$.
Let $m$ be the first index at which a complete column appears, and
translate the second coordinate so that $v_m$ completes column $0$.
At time $m$, no other column is complete.  Put
$T=V(G)\setminus S_m$ and $U=V(G)\setminus N(S_m)$, and define
\[
 T_j=\{i:v_{i,j}\in T\},\qquad
 U_j=\{i:v_{i,j}\in U\},\qquad
 t_j=|T_j|,\qquad u_j=|U_j|.
\]
Thus $t_0=0$ and $1\le t_j\le a$ for $j\ne0$.  A vertex counted by
$u_j$ must have its two horizontal neighbors in the sets counted by
$t_{j-1}$ and $t_{j+1}$.  Its two vertical neighbors must also lie in
$T_j$.  Since $a$ is odd, the map $i\mapsto i+2$ is a single cycle on
$\mathbb Z_a$; hence a nonempty proper $T_j$ contains at most
$t_j-1$ such vertical neighbor-pairs.  Therefore
\[
 u_j\le
 \min\{t_{j-1},t_{j+1},t_j-\mathbf1_{\{0<t_j<a\}}\}.
\]
Since
\[
 D_m=|T|-|U|=\sum_j(t_j-u_j),
\]
Lemma~\ref{lem:cyclic-column-defect} with $q=a$ gives
\[
 D_m\ge\min\{b-1,2a\}.
\]

If $b>2a$, then $D_m\ge2a=\tau$, contradicting
\eqref{eq:terminal-cycle-defect}.  Hence $b\le2a$, so $\tau=b$ and
\[
 b-1\le D_m\le D_k\le b-1.
\]
Equality holds throughout.  The equality clause of
Lemma~\ref{lem:cyclic-column-defect} gives
\begin{equation}\label{eq:adjacent-singleton-holes}
 t_1=t_{b-1}=1.
\end{equation}

Write $v_m=v_{r,0}$ and choose $y\in F_m$.  Then $v_m$ is the unique
neighbor of $y$ in $S_m$.  The witness $y$ cannot lie in column $0$,
because its other $C_a$-neighbor already belongs to $S_{m-1}$.
Thus $y=v_{r,1}$ or $y=v_{r,b-1}$.  In the first case, both
$v_{r-1,1}$ and $v_{r+1,1}$ lie outside $S_m$, so $t_1\ge2$; the second
case similarly gives $t_{b-1}\ge2$.  Both contradict
\eqref{eq:adjacent-singleton-holes}.  Hence $k\le n-\tau$, and
\eqref{eq:grundy-skew-identity} gives the stated lower bound on $Z_-$.
\end{proof}

\subsection{Even tori}

For a bipartite graph $H$ with parts $X$ and $Y$, let
$\gamma_{\mathrm{gr}}^t(H,X)$ denote the maximum length of a legal
open-neighborhood sequence in $H$ whose terms all belong to $X$.  Lin
proved that
\begin{equation}\label{eq:lin-bipartite-one-sided}
 \gamma_{\mathrm{gr}}^t(H)
 =2\gamma_{\mathrm{gr}}^t(H,X)
 =2\gamma_{\mathrm{gr}}^t(H,Y)
\end{equation}
for every bipartite graph $H$ with no isolated vertices
\cite[Proposition~4.2]{LinZeroForcingGrundy}.  We apply this identity to
even tori and bound a sequence supported on one bipartition class by the
cyclic column-defect lemma.

\begin{lemma}\label{lem:one-sided-even-torus-bound}
Let $4\le a\le b$ be even, let $G=C_a\square C_b$, and let
\[
 A_0=\{(i,j):i+j\equiv0\pmod2\},\qquad
 A_1=V(G)\setminus A_0.
\]
If $\mathcal X$ is a legal open-neighborhood sequence whose terms all
belong to one fixed class $A_\varepsilon$, then
\[
 |\mathcal X|\le\frac{ab}{2}-a.
\]
\end{lemma}

\begin{proof}
It is enough to treat $A_0$.  For $j\in\mathbb Z_b$, define the
$A_0$-half-column
\[
 H_j=A_0\cap(\mathbb Z_a\times\{j\}).
\]
Put $q=a/2$, so $|H_j|=q$.  Let $X$ be the set of terms of
$\mathcal X$.  If $X$ contains no complete
$A_0$-half-column, then
\[
 |X|\le b(q-1)
 =\frac{ab}{2}-b
 \le\frac{ab}{2}-a.
\]

Suppose now that a half-column is completed, and let
$\mathcal X_m=(x_1,\ldots,x_m)$ be the first prefix for which this
happens.  Let $j_*$ be the index of the completed half-column.  Apply the
class-preserving automorphism
\[
 (i,j)\longmapsto(i-j_*,j-j_*)
\]
so that $x_m$ completes $H_0$, and put
$X_m=\{x_1,\ldots,x_m\}$.  In the source class $A_0$ and target
class $A_1$, respectively, set
\[
 T=A_0\setminus X_m,\qquad U=A_1\setminus N_G(X_m),
\]
and define
\[
 t_j=|T\cap H_j|,\qquad
 u_j=|U\cap(\mathbb Z_a\times\{j\})|.
\]
Then $t_0=0$ and $1\le t_j\le q$ for $j\ne0$.  Every vertex counted by
$u_j$ has its two horizontal neighbors in the corresponding sets counted
by $t_{j-1}$ and $t_{j+1}$.  In the cyclic order inherited from $C_a$,
the $q$ target vertices in column $j$ determine $q$ consecutive pairs of
vertices of $H_j$; when $q=2$, the same two-element pair occurs twice.
A nonempty proper subset of $H_j$ contains at most one fewer such pair
than vertices, and for $q=2$ it contains none.  Consequently,
\[
 u_j\le
 \min\{t_{j-1},t_{j+1},t_j-\mathbf1_{\{0<t_j<q\}}\}.
\]
Moreover,
\begin{equation}\label{eq:one-sided-half-column-defect}
 |N_G(X_m)|-|X_m|=|T|-|U|=\sum_j(t_j-u_j).
\end{equation}
Lemma~\ref{lem:cyclic-column-defect} gives a lower bound
$\min\{b-1,a\}$ for \eqref{eq:one-sided-half-column-defect}.  If $b>a$,
then $a$ and $b$ are even, so $b\ge a+2$ and this minimum is $a$.

It remains to exclude the value $a-1$ when $b=a$.  If equality held,
the equality clause of Lemma~\ref{lem:cyclic-column-defect} would give
$t_1=t_{b-1}=1$.  Choose a vertex $y$ newly footprinted by the last term
$x_m$.  Thus $x_m$ is the unique neighbor of $y$ in $X_m$.  If $y$ lies
in column $0$, its other vertical neighbor belongs to the completed
half-column $H_0$, a contradiction.  If $y$ lies in column $1$ or
$b-1$, both of its vertical neighbors are outside $X_m$, forcing
$t_1\ge2$ or $t_{b-1}\ge2$, again a contradiction.  Hence
\[
 |N_G(X_m)|-|X_m|\ge a
\]
also when $b=a$.

The defect is nondecreasing along a legal sequence by
\eqref{eq:defect-footprint-excess}.  Therefore
$|N_G(X)|-|X|\ge a$.  Since $N_G(X)\subseteq A_1$ and
$|A_1|=ab/2$, it follows that $|X|\le ab/2-a$.
\end{proof}

\begin{proposition}\label{prop:even-even-cycle-lower-bound}
Let $4\le a\le b$ be even.  Then
\[
 \gamma_{\mathrm{gr}}^t(C_a\square C_b)\le ab-2a,
 \qquad
 Z_-(C_a\square C_b)\ge2a.
\]
\end{proposition}

\begin{proof}
Let $G=C_a\square C_b$, with bipartition classes $A_0$ and $A_1$ as
above.  Lemma~\ref{lem:one-sided-even-torus-bound} and
\eqref{eq:lin-bipartite-one-sided} give
\[
 \gamma_{\mathrm{gr}}^t(G)
 =2\gamma_{\mathrm{gr}}^t(G,A_0)
 \le2\left(\frac{ab}{2}-a\right)
 =ab-2a.
\]
Equation~\eqref{eq:grundy-skew-identity} now gives $Z_-(G)\ge2a$.
\end{proof}

\subsection{Complete classification}

\begin{proof}[Proof of Theorem~\ref{thm:cc-exact}]
Suppose first that $a=b$.  For odd $a$, Bre\v{s}ar et
al.~\cite{BresarEtAlProductGraphs} proved
$\gamma_{\mathrm{gr}}^t(C_a\square C_a)=(a-1)^2$, which is equivalent
to $Z_-(C_a\square C_a)=2a-1$.  For even $a$,
Proposition~\ref{prop:even-even-cycle-lower-bound} and the upper bound in
\eqref{eq:known-cc-interval} give $Z_-=2a$.

Now let $a<b$.  If $a$ is odd,
Corollary~\ref{cor:cycle-forcing-upper-bounds} and
Proposition~\ref{prop:odd-shorter-cycle-lower-bound} give
$Z_-=\min\{b,2a\}$.  If $a$ is even and $b$ is odd, the same corollary
gives $Z_-\le a$, while \eqref{eq:known-cc-interval} gives the reverse
inequality.  Finally, if $a$ and $b$ are even,
Proposition~\ref{prop:even-even-cycle-lower-bound} and
\eqref{eq:known-cc-interval} give $Z_-=2a$.
These cases exhaust $3\le a\le b$, and
\eqref{eq:grundy-skew-identity} gives the corresponding values of
$\gamma_{\mathrm{gr}}^t$.
\end{proof}

%% file: 07_conclusion.tex
\section{Concluding remarks}\label{sec:conclusion}

Theorems~\ref{thm:path-product}, \ref{thm:pc-exact-skew-zero-forcing},
and~\ref{thm:cc-exact} give the Grundy total domination number and the
skew zero forcing number for every Cartesian product with two path or
cycle factors.  Through
$\gamma_{\mathrm{gr}}^t(G)=|V(G)|-Z_-(G)$, forcing sets provide one
direction of each formula, while matrix nullity or direct bounds on legal
open-neighborhood sequences provide the other.

The connection established in earlier work between greedy neighborhood
sequences, zero forcing, and minimum-rank problems is essential to the
classification.  It turns questions that initially appear to belong to
different areas into mutually useful formulations: forcing dynamics
suggest candidate extremal sets, matrix spectra certify large nullspaces,
and footprint counting controls the longest legal sequences.  The present
results illustrate that the connection is most effective when it permits
techniques from these areas to be combined, rather than when one
formulation is used exclusively.

For path--path and path--cycle products, Kronecker differences attain the
forcing values and also determine maximum skew nullity and minimum skew
rank.  Their only strict skew-nullity gap occurs when $p$ is even, $c$ is
odd, and $p\ge c$:
\[
 M_-(P_p\square C_c)=c-1
 <c=Z_-(P_p\square C_c)=M_0(P_p\square C_c).
\]
Nullity parity prevents a skew-symmetric certificate from attaining this
odd forcing value, while a hollow symmetric certificate does attain it.
Thus the linear-algebraic part proves more than the corresponding defect
bound would provide: throughout the path-containing families it identifies
the extremal matrix nullity and hence the minimum skew rank.  At the same
time, the strict gap shows exactly why skew-symmetric spectral certificates
alone cannot establish the entire forcing classification.

The previously unresolved odd-shorter and even-even lower bounds are
obtained from a common column-defect argument.  For an odd shorter cycle,
the first completed column is charged with $q=a$; the equality case
supplies the final unit in the short range.  For two even cycles, Lin's
one-sided identity reduces the problem to a single bipartition class, and
the same cyclic height estimate is applied with $q=a/2$ to the first
completed half-column in that class.  The resulting loss is at least $a$
per class and hence $2a$ in total.  This columnwise refinement of earlier
footprint and arc counting works uniformly for all admissible cycle orders
in these branches.  It is not, however, a universal replacement for the
matrix method: its sharpness depends on finding a completed column or
half-column and on the parity structure that controls the corresponding
cyclic height sequence.  Recognizing such a layer is itself the difficult
step, and no comparable defect profile is presently available that would
recover the matrix-nullity conclusions for all path-containing products.

The matching upper bounds are likewise structural rather than automatic.
A full boundary layer gives some elementary sweeps, but the sharp
path--cycle construction for an odd cycle begins from a sparse step-two
pattern and must close correctly at the cyclic seam.  The closed-diagonal
and alternating half-layer constructions on tori similarly use the parity
and periodicity of the factors.  These examples show that even after a
lower-bound mechanism is known, finding a minimum skew zero forcing set
can require a separate geometric idea.

The formulas also show that $Z_-$ is not monotone as a function of a
factor order.  For example,
\[
 Z_-(C_4\square C_5)=Z_-(C_4\square C_7)=4,
 \qquad Z_-(C_4\square C_6)=8,
\]
and
\[
 Z_-(P_3\square C_5)=Z_-(P_3\square C_7)=3,
 \qquad Z_-(P_3\square C_6)=6,
\]
whereas $Z_-(P_3\square P_3)=3>2=Z_-(P_3\square P_4)$.  These jumps
reflect the parity dependence of both the forcing constructions and the
spectral certificates.

This complementarity also explains why the exact values remained open
after the earlier bounds.  In some regimes the decisive obstruction is
spectral, in others it is encoded by the first completed periodic layer,
and the sharp forcing set may use a third feature of the product geometry.
Bounds obtained from any one of these viewpoints need not meet.  The full
classification becomes possible only after the appropriate certificate,
defect calculation, and forcing construction are matched in each parity
regime.

%% file: 08_declarations.tex
\section*{Funding}

This work was supported by the \ManuscriptFundingAgency\ under Grant
\ManuscriptGrantNumber.

\section*{Data availability}

No datasets were generated or analyzed in this work.  Auxiliary
computational checks were used only as diagnostic support and are not
required for any proof.

\section*{Declaration of competing interest}

The author declares that there are no competing interests.

\section*{Declaration of generative AI and AI-assisted technologies in the
manuscript preparation process}

During the preparation of this work, the author used ChatGPT and Codex
(OpenAI) and Claude Opus 5 (Anthropic) as auxiliary tools.  Specifically,
these tools assisted with exploring candidate arguments, testing small
instances, simulating proposed skew-forcing sequences, checking matrix and
defect calculations, identifying gaps and boundary cases in draft proofs,
organizing the exposition, checking relevant references and journal
requirements, and editing English and LaTeX.  The research questions,
conjectures, core constructions, and central mathematical ideas originated
with the author, who directed the investigations, selected the arguments
included in the manuscript, and verified the final proofs.  Computational
checks were used only as diagnostic support and do not replace any proof.
All AI-assisted output was critically reviewed, revised, or discarded by
the author, who takes full responsibility for the originality, correctness,
and content of the article.